\documentclass{amsart}
\usepackage{amsmath, amssymb, amsfonts, url, color, hyperref}
\newtheorem{theorem}{Theorem}[section]

\newtheorem{proposition}{Proposition}[section]

\newtheorem*{question*}{Question}
\newtheorem{definition}{Definition}[section]
\newtheorem{example}{Example}[section]

\newtheorem*{acknowledgements*}{Acknowledgements}
\numberwithin{equation}{section}

\newcommand{\CC}{\mathbb{C}}
\newcommand{\RR}{\mathbb{R}}

\begin{document}

\title{Localization of Forelli's theorem}

\author{Ye-Won Luke Cho}
\begin{abstract}
	The main purpose of this article is to present a localization of Forelli's 
	theorem for functions holomorphic along a standard suspension of linear discs. This generalizes one of the main results of \cite{CK21} and the original Forelli's theorem.
\end{abstract}


\maketitle

\section{Introduction}
\subsection{Notations and terminology}
Denote by $B^n(a;r) := \{z \in \CC^n \colon \|z-a\|<r  \}$ and by
$S^m := \{v \in \RR^{m+1} \colon \|v\|=1\}$.  With such notation,
the boundary of $B^n:=B^n(0;1)$ is $S^{2n-1}$.  For a nonempty subset $F$ of $S^{2n-1}$, consider $S_0(F):=\{zv: z\in B^1, v\in F\}$. By a $\textit{\textup{(}standard\textup{)} suspension}$, we mean a pair $(S_0(F),\psi)$, where $\psi(u, \lambda) := \lambda u$ for each $\lambda\in B^1,~u\in F$. For simplicity, we will denote a suspension by its underlying set. We say that $S_0(F)$ is a $\textit{formal Forelli suspension}$ if any function $f:B^n\to \mathbb{C}$ satisfying the following two conditions
\begin{enumerate}
	\setlength\itemsep{0.3em}
	\item $f\in C^{\infty}(0)$, i.e., for each positive integer $k$ there
	exists an open neighborhood $V_k$ of the origin $0$ such that $f \in C^k 
	(V_k)$, and
	\item $f$ is holomorphic along $S_0(F),$ i.e., $z\in B^1\to f(zv)$ is holomorphic for each $v\in F$
\end{enumerate}
has a formal Taylor series $S_f=\sum C_{I}^{J}z^{I}\bar{z}^{J}$  of holomorphic type, that is, $C_{I}^{J}=0$ whenever $J\neq 0$. See (\ref{formal Taylor series}) for the definition of $C_{I}^{J}$.  We also say that $S_0(F)$ is a $\textit{normal suspension}$ if the following is true: any formal power series $S\in \mathbb{C}[[z_1,\dots,z_n]]$ for which $S_v(t):=S(tv)\in \mathbb{C}[[t]]$ is holomorphic for each $v\in F$ converges uniformly on a neighborhood of the origin.  A formal Forelli suspension that is also normal is called a $\textit{Forelli suspension}$. Then any function $f:B^n\to \mathbb{C}$ that is smooth at the origin and holomorphic along a Forelli suspension must be holomorphic on $B^n(0;r)$ for some $r>0$.

  Fix a suspension $S_0(F)$ and denote by $\bar{F}$ the closure of $F$ in $S^{2n-1}$. $S_0(F)$ is said to have an $\textit{algebraically nonsparse leaf}~ L_v:=\{zv:z\in B^1\}$ generated by $v\in \bar{F}$ if the following is true: for each open neighborhood $U\subset S^{2n-1}$ of $v$ and a real polynomial $P\in \mathbb{R}[x_1,y_1,\dots,x_n,y_n]$ satisfying
 \[
  S_0(U\cap \bar{F})\subset Z(P):=\{z\in \mathbb{C}^n:P(z)=0\},
 \] we have $P\equiv 0$ on $\mathbb{C}^n$. In this case, the suspension is said to be $\textit{nonsparse}$. A suspension is $\textit{sparse}$ if it has no nonsparse leaf. For each $i\in \{1,\dots,n\}$, define 
 \begin{align*}
	\begin{gathered}
		F'_i:= \{(\frac{z_1}{z_i}, \ldots,\frac{z_{i-1}}{z_i},\frac{z_{i+1}}{z_i},\ldots, \frac{z_n}{z_i})\in \mathbb{C}^{n-1}
		\colon  (z_1,\ldots,z_n)\in F, z_i\neq 0 \},~\text{and}\\
		F':=\bigcup_{i=1}^{n}F'_i.
	\end{gathered}
\end{align*}
 $S_0(F)$ is said to have a $\textit{regular leaf}~L_v$ generated by $v=(v_1,\ldots,v_n)\in \bar{F}$ if there exists an integer $i$ $\in \{1,\dots,n\}$ such that $v_i\neq 0$ and the set $F'_i$ is locally $L$-regular at $v'_i:=(\frac{v_1}{v_i}, \ldots,\frac{v_{i-1}}{v_i},\frac{v_{i+1}}{v_i},\ldots,$ $ \frac{v_n}{v_i})$. For the definition of $L$-regularity, see Section \ref{Sect 3}. A suspension is called a $\textit{regular suspension}$ if it has a regular leaf.
 
 \subsection{Main theorem} In this paper, we prove the following 
 \begin{theorem}\label{main theorem}
 	If a suspension $S_0(F)$ has a nonsparse leaf and a regular leaf, then it is a Forelli suspension; that is, any function $f:B^n\to \mathbb{C}$ satisfying the following two conditions
 	\begin{enumerate}
 	\item $f\in C^{\infty}(0)$, and
 	\item $z\in B^1\to f(zv)$ is holomorphic for each $v\in F$
 	\end{enumerate}
 is holomorphic on $B^n(0;r)$ for some $r>0$. Furthermore, there exists a neighborhood $U\subset S^{2n-1}$ of a generator $v_0\in \bar{F}$ of the regular leaf such that $f|_{B^n(0;r)}$ extends to a holomorphic function on $B^n(0;r)\cup S_0(U)$.
 \end{theorem}
As each point of an open subset of $S^{2n-1}$ generates a leaf that is both nonsparse and regular, Theorem \ref{main theorem} implies the following 
\begin{theorem}[Cho-Kim \cite{CK21}] \label{CK21}
Let $U$ be a nonempty open subset of $S^{2n-1}$. If $f:B^n\to \mathbb{C}$ satisfies the following two conditions
	\begin{enumerate}
	\item $f\in C^{\infty}(0)$, and
	\item $z\in B^1\to f(zv)$ is holomorphic for each $v\in U$,
\end{enumerate}
then $f$ is holomorphic on $B^n(0;r)\cup S_0(U)$ for some $r>0$.
\end{theorem}

Note that Theorem \ref{CK21} reduces to the classical Forelli's theorem \cite{Forelli77} when $U=S^{2n-1}$. At this point, we would like to give two remarks regarding Theorem \ref{main theorem}. First, the analyticity of the given function $f$ depends on the local behavior of $f$ near the two specific leaves of $S_0(F)$. Therefore, the theorem can be regarded as a localization of Forelli's theorem. Second, in contrast to the case of Theorem \ref{CK21}, the suspension in Theorem \ref{main theorem} needs not to be generated by an open subset of $S^{2n-1}$ in general; we will construct a nowhere dense Forelli suspension in Example \ref{nowhere dense Forelli suspension} of Section \ref{Sect 4}.

\subsection{Structure of paper, and remarks}
Forelli's theorem  has been generalized to various directions and several works in this line of research were motivated by \cite{Chirka06} (see, for example, \cite{KPS09}, \cite{JKS13}, \cite{JKS16}, \cite{Krantz18}, \cite{Sadullaev}, \cite{CK21},  just to name a few). Chirka proved in \cite{Chirka06} that the set of complex lines in \cite{Forelli77} can be replaced with a singular foliation of $B^2-\{0\}$ by holomorphic curves transversal at the origin to obtain the same conclusion. This work was generalized further to a higher-dimensional principle \cite{JKS13} and finally by the author and K.-T. Kim to the case where the foliation is parametrized by an open subset of $S^{2n-1}$ (Theorem 1.3 in \cite{CK21}).  Theorem 1.3 \cite{CK21} in particular answered the second question in Section 6  of \cite{Chirka06} (p.\ 219).
 
The proof of Theorem 1.3 in \cite{CK21} can be reduced to the case of standard suspension (Theorem \ref{CK21}) by applying the local analysis of \cite{JKS13} and the Baire category theorem. Then the proof of Theorem \ref{CK21} was settled in two steps:
\bigskip
\begin{itemize}
		\setlength\itemsep{0.1em}
\item [(1)] The formal Taylor series $S_f$ of $f$ is of holomorphic type.
\item [(2)] The formal series $S_f$ converges uniformly on $B^n(0;r)$ for some $r>0$.\\
\end{itemize}
\noindent
Then $f\equiv S_f$ on $B^n(0;r)$ and moreover, $f$ becomes holomorphic on $B^n(0;r)\cup S_0(U)$ by Hartogs' lemma (see Prop.7 in p.41 of \cite{Nara95}). Although the openness of the set $U\subset S^
{2n-1}$ was used in Step (1), it turned out that $U$ only needs to have a $\textit{nonpluripolar}$ direction set $U'$ to guarantee the uniform convergence of $S_f$ in Step (2). Therefore, several researchers including  K.-T. Kim, A. Sadullaev were led to ask whether the openness condition on $U$ can be weakened to obtain a further generalization of Forelli's theorem. The purpose of this paper is to provide an answer (Theorem \ref{main theorem}) to the question. On the other hand, we do not know how to generalize Theorem \ref{main theorem} to the case of suspension of nonlinear Riemann surfaces.

 In this paper, we proceed to prove Theorem \ref{main theorem} by following the aforementioned steps. In Section \ref{Sect 2}, we characterize formal Forelli suspensions with a mild adjustment on formal power series analysis of \cite{CK21}. In Section \ref{Sect 3}, we show that the local characterization of normal suspensions follows from the results of \cite{LevenMol88}, \cite{CK21}, and the pluripotential theory. See also \cite{Sadullaev}. We also give the proof of Theorem \ref{main theorem} in this section. Finally, we construct several examples of suspensions in Section \ref{Sect 4} which in particular indicate that a formal Forelli suspension is not necessarily normal, nor vice versa.

\subsection*{Acknowledgements}
The author would like to thank professor Kang-Tae Kim for various helpful discussions, professor Azimbay Sadullaev for his suggestion on a generalization of Theorem \ref{CK21}, and Seungjae Lee for helpful comments. The author also wishes to express his deep gratitute to the referees for their valuable comments on the article. Most parts of the paper were written while the author was supported by the National Research Foundation of Korea (NRF-4.0019528) at Pohang University of Science and Technology. The paper was completed at Pusan National University where the author is currently supported by the National Research Foundation of Korea (NRF-2018R1C1B3005963,  NRF-2021R1A4A1032418).
\section{Formal Forelli Suspensions}\label{Sect 2}

In this section, we settle the following 
\begin{theorem}\label{formal Forelli suspension theorem}
	A suspension is a formal Forelli suspension if, and only if, it has a nonsparse leaf.
\end{theorem}
Let $z=(z_1,\dots,z_n)=(x_1,y_1,\dots,x_n,y_n)$ be the standard complex coordinate system on $\mathbb{C}^n$, where $z_k=x_k+iy_k$ for each $k\in \{1,\dots,n\}$. Recall the multi-index notation as follows: 
\begin{align*}
	\alpha=(\alpha_1,\ldots,\alpha_n),~|\alpha|=\alpha_1+\cdots+\alpha_n,~\alpha!=\alpha_1!\cdots \alpha_n!,~\text{and}~z^{\alpha}=z_1^{\alpha_1} \cdots z_n^{\alpha_n}.
\end{align*}
We also define the following differential operators
\begin{equation}\label{differentiation}
	\frac{\partial^{|\alpha|}f}{\partial z^\alpha}:=\frac{\partial^{\alpha_1+\cdots+\alpha_n}f}{\partial z_1^{\alpha_1}\cdots \partial z_n^{\alpha_n}}, ~\frac{\partial^{|\alpha|}f}{\partial \bar{z}^\alpha}:=\frac{\partial^{\alpha_1+\cdots+\alpha_n}f}{\partial \bar{z}_1^{\alpha_1} \cdots \partial \bar{z}_n^{\alpha_n}},
\end{equation}
whenever $f\in C^{|\alpha|}(\Omega)$ for some open set $\Omega \subset \mathbb{C}^n$. If $S=\sum C_{I}^{J}z^{I}\bar{z}^{J}\in \mathbb{C}[[z_1,\dots,$ $z_n]]$, then we define 
\[
\frac{\partial^{|\alpha|}S}{\partial z^\alpha}:=\sum C_{I}^{J}\cdot \frac{\partial^{|\alpha|}z^{I}}{\partial z^\alpha}\cdot \bar{z}^{J}, ~\frac{\partial^{|\alpha|}S}{\partial \bar{z}^\alpha}:=\sum C_{I}^{J}\cdot \frac{\partial^{|\alpha|}\bar{z}^{J}}{\partial \bar{z}^\alpha}\cdot {z}^{I}.
\]

 We first prove that $S_0(F)$ is not a formal Forelli suspension under the assumption that $S_0(F)$ is sparse. Then for each $v\in \bar{F}$, there exist an open neighborhood $U_v\subset S^{2n-1}$ of $v$ and a nonconstant polynomial $P_v\in \mathbb{R}[x_1,y_1,\dots,x_n,y_n]$ such that $P_v\equiv 0$ on $S_0(\bar{F}\cap U_v)$. In complex coordinates, $P_v=\sum C_{I}^{J}z^{I}\bar{z}^{J}$ for some complex numbers $C_{I}^{J}$ that depend on $v$. Here, $C_{I}^{J}\neq0$ for some $J\neq 0$ as $P_v$ is a nonconstant real polynomial. So each $P_v$ is not of holomorphic type. Since $\mathcal{U}:=\{U_v:v\in F\}$ is an open cover of the compact set $\bar{F}$, there exists a finite subcover $\{U_{v_1},\dots,U_{v_m}\}$ of $\mathcal{U}$. Then the polynomial $P:=P_{v_1}\cdot P_{v_2}\cdots P_{v_m}$ is smooth and holomorphic along $S_0(F)$ but it is not of holomorphic type.
 
Conversely, suppose that $S_0(F)$ has a nonsparse leaf and let $f:B^n\to \mathbb{C}$ be a function satisfying the following two conditions:
\begin{enumerate}
	\item $f\in C^{\infty}(0)$, and
	\item $f$ is holomorphic along $S_0(F)$.
\end{enumerate}
Then we are to show that $f$ has a formal Taylor series $S$ of holomorphic type. Recall that $S$ is defined by
\[
S:=\sum C_{I}^{J}z^I\bar{z}^{J},
\]
where
\begin{equation}\label{formal Taylor series}
C_{I}^{J}:=\frac{1}{I!J!}\frac{\partial^{|I|+|J|}f}{\partial z^I\partial \bar{z}^{J}}(0).
\end{equation}

Then it is straightforward to check that $S(zv)=f(zv)$ for each $z\in B^1,~v\in F$. So $z\in B^1\to S(zv)$ is also convergent and holomorphic for each $v\in F$. This implies that 
\begin{equation}\label{Formal holomorphicity}
	\bar{E}S\equiv 0~\text{on}~S_0(F), 
\end{equation}
where 
\[
E := \sum\limits_{k=1}^n z_k \frac\partial{\partial z_k}.
\]

We first prove the claim in the case that $S$ is a polynomial of finite 
degree. Note that (\ref{Formal holomorphicity}) and the continuity of $\bar{E}S$ implies
\[
\textup{Re}(\bar{E}S)\equiv 0~\text{and}~ \textup{Im}(\bar{E}S)\equiv 0~\textup{on}~S_0(\bar{F}).
\]
Then, as $S_0(F)$ has a nonsparse leaf, it must be that 
\[
\textup{Re}(\bar{E}S)=\textup{Im}(\bar{E}S)\equiv 0~\text{on}~\mathbb{C}^n.
\]
Therefore, $\bar{E}S\equiv 0$ on $\mathbb{C}^n$. Consider the monomial term $S_{IJ}$ in $S$ of a fixed multi-degree 
$(I, J) := (i_1,\ldots,i_n,j_1,\ldots,j_n)$.
The monomial term in $\bar{E}S$ of multi-degree $(I, J)$ is precisely 
$\bar E (S_{IJ})$.   Then $\bar{E}S \equiv 0$ implies
\[
|J|C_{i_1,\ldots,i_n}^{j_1,\ldots,j_n}=0.
\] 
Consequently, $C_{i_1,\ldots,i_n}^{j_1,\ldots,j_n}=0$, whenever 
$J=(j_1,\ldots,j_n) \neq 0$.

Now let $S=\sum C_{I}^{J}z^I\bar{z}^J$ be any formal power series. Then, for each positive integer $m$, 
we have 
\[
\bar{E}(S_m)=(\bar{E}S)_m=0~\text{on}~S_0(F),
\] 
where 
\[
S_m:=\sum_{|I|+|J|= m}C_{I}^{J}z^I\bar{z}^J.
\] 
Thus we conclude from the previous arguments on finite polynomials that 
$S$ is of holomorphic type.\hfill $\Box$
\section{Normal Suspensions}\label{Sect 3}
We denote by $\textup{PSH}(\mathbb{C}^n)$ the set of all plurisubharmonic functions defined on $\mathbb{C}^n$. A set $E\subset \mathbb{C}^n$ is called $\textit{pluripolar}$ if there exists a nonconstant function $u\in \textup{PSH}(\mathbb{C}^n)$ such that $E \subset \{z\in \mathbb{C}^n: u(z)=-\infty\}$. 

\begin{definition}[\cite{Siciak81}]
	\textup{ For each positive 
		integer $n$ and a subset $E$ of $\mathbb{C}^n$, define
		\[ 
		\mathcal{L}_n:=\{u\in \textup{PSH}(\mathbb{C}^n):\exists C_u\in \mathbb{R}
		~ \text{such that}~u(z)\leq C_u+\textup{log}\,(1+\|z\|)
		~ \forall z\in \mathbb{C}^n \}, 
		\]  
		\[
		V_E(z):=\textup{sup}\,\{u(z)\colon u\in \mathcal{L}_n ,u\leq 0 ~\text{on}~ E\},
		~ \forall z\in \mathbb{C}^n.
		\]
		Then the $\textit{logarithmic capacity}$ of $E$ is defined to be 
		\[
		c_n(E):=e^{-\gamma_n(E)},
		\]
		where
	}
	\begin{align*}
		\gamma_n(E)&:=\limsup\limits_{\|z\|\to \infty}\,(V^*_E(z)-\textup{log}\,\|z\|),\\
		V^*_E(z)&:=\limsup\limits_{w\to z}V_E(w).
	\end{align*}
\end{definition}
\noindent 
Then it follows from Theorem 3.10 in \cite{Siciak81} that a set $E\subset \mathbb{C}^n$ is pluripolar if and only if $c_n(E)=0$. Now we can state the following
\begin{theorem}[\cite{CK21}]\label{CK Lelong}
A suspension $S_0(F)$ in $\mathbb{C}^n$ is normal if $c_{n-1}(F')\neq 0$.
\end{theorem}
For a historical account on the theorem, see Section 3.2 in \cite{CK21}. Although this result was also known earlier (see, for example, \cite{LevenMol88}, \cite{Sadullaev}), the method in \cite{CK21} is more elementary and different from the previous works.  The crux of the proof of Theorem $\ref{CK Lelong}$ turns out to be the following higher-dimensional generalization of Lelong's theorem (see Theorem 2 in \cite{Lelong51} and Prop.4.1 in \cite{CK21}).
\begin{proposition} [\cite{CK21}] \label{estimate}
Let $\{P_k\}\subset \mathbb{C}[z_1,\ldots,z_n]$ be a sequence of 
polynomials with $\textup{deg}\,P_k\leq k$ for each positive integer $k$. If $F\subset \mathbb{C}^n$ is a set satisfying $c_n(F)\neq 0$ and
\[
\limsup\limits_{k\to \infty}\frac{1}{k}\, \textup{log}\,|P_k(z)|<\infty\, \text{for each}~ z\in F,
\]
Then, for each compact subset $K$ of $\mathbb{C}^n$, there exists a constant $M=M(K)>0$ such that $\frac{1}{k}\,\textup{log}\,|P_k(z)|<\textup{log}\,M$ for any $z\in K$, $k\geq 1$. 
\end{proposition}

We sketch the proof of Theorem \ref{CK Lelong} for the case of $n=2$, assuming Proposition \ref{estimate}. Let 
\[
S(z_1,z_2)= \sum a_{i,j}{z_1}^{i}{z_2}^{j}
\]
 be a formal power series for which $S_{a_1,a_2}(t):=S(a_1t,a_2t)$ has a positive radius of 
convergence $R_{(a_1,a_2)}>0$ for every $(a_1,a_2)\in F$. Without loss of generality, assume that $c_1(F'_1)\neq 0$. Then we are to show that $S$ is holomorphic on some open neighborhood 
of $0$ in $\mathbb{C}^2$.  Note that, for any $b\in F'_1$, $S_{(1,b)}$ converges absolutely and uniformly on $\frac{1}{2}R_{(1,b)}$. So it can be rearranged as follows:
\[
S_b(t):=S(t,bt)=\sum_{n=0}^{\infty}(\sum_{j=0}^{n}a_{n-j,j}b^j)t^n = \sum_{n=0}^{\infty}P_n(b)t^n,
\]
where
\[
P_n(z):=\sum_{j=0}^{n}a_{n-j,j}z^j\in \mathbb{C}[z].
\]
 The root test implies
\[
\limsup\limits_{n\to \infty}\frac{1}{n}\,\textup{log}\,|P_n(b)|<\infty\,\text{for each} ~b\in F'_1.
\]
 Recall that $c_1(F'_1)\neq 0$. So by Proposition \ref{estimate}, there exists a constant $M>0$ satisfying
\[
\frac{1}{n}\, \textup{log}\,|P_n(b)|<M
 \]
 or equivalently,
 \[
 |P_n(b)|<M^n\,\text{for all}\,b\in S^1.
 \] 
 The Cauchy estimate implies that, for each positive integer $n$ and $0\leq k\leq n$, we have $|a_{n-k,k}|<M^n$. Then
\[
\sum_{i+j=n}
|a_{i,j}{z_1}^{i}{z_2}^{j}|<n\cdot 2^{-n}
\] 
whenever $|z_1|,|z_2|<\frac{1}{2M},~n\geq 1.$ Then we conclude from Weierstrass' $M$ test that $S$ is uniformly convergent on a neighborhood of the origin. \hfill $\Box$

The following theorem provides a partial converse to Theorem \ref{CK Lelong}.
\begin{theorem}[\cite{LevenMol88}]\label{Levenberg Molzon}
A suspension $S_0(F)$ is not normal if $F'$ is a $F_{\sigma}$ set satisfying $c_{n-1}(F')=0$; that is, there exists a formal series $S\in \mathbb{C}[[z_1,\dots,z_n]]$ such that $z\in B^1\to S(zv)$ converges for each $v\in F$ but $S$ does not converge uniformly on any open neighborhood of the origin.
\end{theorem}

We say that a set $E\subset \mathbb{C}^n$ is $\textit{L-regular}$ at $a\in \bar{E}$ if $V^{\ast}_E(a)=0$. $E$ is said to be $\textit{locally L-regular}$ at $a\in \bar{E}$ if $E\cap B^n(a;r)$ is $L$-regular for each $r>0$. We remark that $E$ is nonpluripolar if, and only if, $E$ is locally $L$-regular at some point; if a set $E\subset \mathbb{C}^n$ is locally $L$-regular at $a\in \bar{E}$, then $E$ is nonpluripolar as $V^{\ast}_{E}\equiv +\infty$ on $\mathbb{C}^n$ whenever $E$ is pluripolar. Conversely, if there is no point at which $E$ is locally $L$-regular, then $E$ is pluripolar as the set
\[
\{z\in \bar{E}:\,E~\text{is not locally} ~L\text{-regular at}~z\}
\]
is known to be always pluripolar. See p.186 of \cite{Klimek91}.  Then the discussions in this section are summarized in the following 
\begin{theorem}\label{local characterization of Normal suspensions}
A suspension $S_0(F)$ is normal if it has a regular leaf. Conversely, if $F'$ is an $F_{\sigma}$ set and $S_0(F)$ has no regular leaf, then $S_0(F)$ is not normal.
\end{theorem}
\begin{proof}
If $S_0(F)$ has a regular leaf, then $F'$ is $L$-regular at some point so it is nonpluripolar, i.e., $c_{n-1}(F')=0$. Therefore, $S_0(F)$ is a normal suspension by Theorem \ref{CK Lelong}. Conversely, if $S_0(F)$ has no regular leaf, then $F'$ is pluripolar. So it follows from Theorem \ref{Levenberg Molzon} that $S_0(F)$ is not normal if $F'$ is an $F_{\sigma}$ set and $S_0(F)$ has no regular leaf.
\end{proof}
\textit{Proof of Theorem \ref{main theorem}}. Let 
$f\colon B^n\to \mathbb{C} $ be a function that is smooth at the 
origin and holomorphic along a Forelli suspension $S_0(F)$. Then the 
formal Taylor series $S_f$ is of holomorphic type by Theorem \ref{formal Forelli suspension theorem}. Note that $S_f$ converges uniformly on $B^n(0;r)$ for some $r>0$ by Theorem \ref{local characterization of Normal suspensions}. Now $f=S_f$ is holomorphic on $B^n(0;r)$ and moreover, Shiffman's generalization of Hartogs' lemma \cite{Shiffman89} implies that there exist an open neighborhood $U\subset S^{2n-1}$ of the regular leaf of $S_0(F)$ such that $f|_{B^n(0;r)}$ extends to a holomorphic function defined on $B^n(0;r)\cup S_0(U)$. \hfill $\Box$
\section{Examples of Suspensions}\label{Sect 4}

\begin{example}\label{rational suspension}
\normalfont 
Fix a nonempty open subset $U$ of $S^{2n-1}$ and define 
\[
F:=\{(z_1,\dots,z_n)\in U: z_i\in \mathbb{Q}~\forall i\in \{1,...,n\}\}.
\]
 Then $F'$ is countable and therefore it is a pluripolar $F_{\sigma}$ set. Hence $S_0(F)$ is not normal. It is straightforward to check that $S_0(F)$ is nonsparse as $S_0(\bar{F})=S_0(U)$ has a nonempty interior and every polynomial in $\mathbb{R}[x_1,y_1,\dots,x_n,y_n]$ is real-analytic. Therefore, $S_0(F)$ is a dense formal Forelli suspension.
\end{example}

In the following, we identify $\mathbb{R}^{2n-1}$ with the set $\{(z_1,\dots,z_n)\in \mathbb{C}^n:  \textup{Im}\,z_1=0\}.$ 
\begin{example}\label{nowhere dense formal Forelli suspension}
	\normalfont 
	We construct a nowhere dense formal Forelli suspension which is not normal.
	Let $\{r_{k}\}$, $\{s_\ell\}\subset \mathbb{R}$ be two sequences that decreases from $\frac{\pi}{4}$ to $0$, increases from $0$ to $\frac{\pi}{2}$, respectively. Fix $v:=(1,0)\in S^3$ and for each positive integer $\ell$, define
	\begin{align*}
	\begin{gathered}
	F_{\ell}:=\{(\textup{cos}\,r_{k},\,e^{is_\ell}\textup{sin}\,r_{k}) \in   \mathbb{R}^3\cap S^3:k~ \text{is a positive integer}\},\\
	F:=\bigcup_{\ell=1}^{\infty}{F}_{\ell}.
	\end{gathered}
	\end{align*}
  Note that $S_0(F)$ is not normal as $F'$ is countable. We prove that $S_0(F)$ is a formal Forelli suspension by showing that $v\in \bar{F}$ generates a nonsparse leaf. 
	
	Fix an open neighborhood $U$ of $v$ in $S^3$ and suppose that $S_0(\bar{F}\cap U)\subset Z(P)$ for some real polynomial $P\in \mathbb{R}[x_1,y_1,x_2,y_2]$. Then $P$ can be written as
	
	\begin{equation}\label{sum}
		P(z_1,\bar{z}_1,z_2,\bar{z}_2)=\sum_{0\leq \alpha,\beta,\gamma,\delta\leq N}C_{\alpha\beta}^{\gamma\delta}\cdot z_1^{\alpha}z_2^{\beta}\bar{z}^{\gamma}_1\bar{z}^{\delta}_2,
	\end{equation}
 where $\{C_{\alpha\beta}^{\gamma\delta}\}$ is a finite set of complex numbers. Now we are to show that $P\equiv 0$ on $\mathbb{C}^2$. As $v$ is a limit point of each $F_k$, there exists a positive integer $M$ such that
 \[
  (z\,\textup{cos}\,r_{k},\,ze^{is_\ell}\textup{sin}\,r_{k})\in S_0(\bar{F}\cap U)\subset Z(P)
  \] 
  if $k,\ell\geq M,$ and $z\in B^1$. Then
	\begin{align}	\label{Equation}
	0&=P(z\,\textup{cos}\,r_{k},\,ze^{is_\ell}\textup{sin}\,r_{k}) \nonumber \\
&=\sum_{0\leq \alpha,\beta,\gamma,\delta\leq N} \big\{C_{\alpha\beta}^{\gamma\delta}\cdot (\textup{cos}\,r_{k})^{\alpha+\gamma}(\textup{sin}\,r_{k})^{\beta+\delta}e^{is_\ell(\beta-\delta)}\big\}z^{\alpha+\beta}\bar{z}^{\gamma+\delta}.
	\end{align}
	Fix nonnegative integers $m,p,q,r$. Since each coefficient of the monomial of multi-degree $(p,q)$ in $(\ref{Equation})$ vanishes, we have
	\begin{equation} \label{Equation 2}
		\sum_{0\leq \alpha,\beta,\gamma,\delta\leq N}\big\{C_{\alpha\beta}^{\gamma\delta}\cdot (\textup{cos}\,r_{k})^{\alpha+\gamma}(\textup{sin}\,r_{k})^{\beta+\delta}e^{is_\ell(\beta-\delta)}\big\}=0
	\end{equation}
  whenever $k,\ell \geq M,~\alpha+\beta=p,$ and $\gamma+\delta=q$. Note that $(\ref{Equation 2})$ is equivalent to the following equation
	\begin{equation*}
		g(z)=\sum_{n=-N}^{N}P_n(\textup{cos}\,r_{k},\, \textup{sin}\,r_{k})\,z^n=0~\forall z\in \{e^{is_\ell}\},
	\end{equation*}
	where 
	\[
	P_n(x,y)=\sum_{\substack{\alpha+\beta=p \\ \gamma+\delta=q\\ \beta-\delta =n}}  C_{\alpha\beta}^{\gamma\delta}\cdot x^{\alpha+\gamma}y^{\beta+\delta}
	\]
    is a homogeneous polynomial of degree $p+q$ in real variables $x,y$. As $g$ is holomorphic on $\mathbb{C}-\{0\}$, it follows from the identity theorem that $P_n(1,\textup{tan}\,r_{k})$ $= 0$ for each integer $k$ and $n$. Then finally, the coefficient of each monomial in $P_n(1,t)$ is zero since $P_n(1,t)\in \mathbb{C}[t]$ is a finite polynomial. Therefore
	\[
	\sum C_{\alpha\beta}^{\gamma\delta} = 0,
	\]
	where the sum is taken over all quadruple $(\alpha,\beta,\gamma,\delta)$ satisfying
	\begin{equation}\label{linear equation}
		\begin{cases}
			\beta-\delta = m \\
			\alpha+\beta=p \\
			\gamma+\delta = q \\
		    \beta+\delta = r.		
		\end{cases}
	\end{equation}
	Note that (\ref{linear equation}) always has a unique solution. So we have $C_{\alpha\beta}^{\gamma\delta}=0$ for any quadruple $(\alpha,\beta,\gamma,\delta)$ appearing in (\ref{sum}). Then $P\equiv 0$ on $\mathbb{C}^2$ as desired.
	
	This construction can be generalized to higher dimensions. Let $x_1(\theta)=\textup{cos}\,\theta,$ $x_2(\theta)=\textup{sin}\,\theta$ be the parametrization of $S^1$ and define a parametrization of $S^{n+1}$ inductively as
		\begin{equation*}
		\begin{cases}
			x_i(\theta_1,\dots,\theta_{n},\theta_{n+1}) =x_i(\theta_1,\dots,\theta_{n})\cdot \textup{cos}\,\theta_{n+1} ~\text{for}~1\leq i \leq n+1,\\
			x_{n+2}(\theta_1,\dots,\theta_{n},\theta_{n+1}) = \textup{sin}\,\theta_{n+1},
		\end{cases}
	\end{equation*}
     where $\{x_i(\theta_1,\dots,\theta_n):1\leq i\leq n+1\}$ is the parametrization of $S^{n}$ chosen in the previous induction step. Fix two $(n-1)$-tuples $k:=(k_1,\dots,$ $k_{n-1}),\,\ell:=(\ell_1,\dots,\ell_{n-1})$  of positive integers and define
		\begin{gather*}
			x_i(k):=x_i(r_{k_1},\dots,r_{k_{n-1}})~\forall i\in \{1,\dots,n\},\\
			F^n_{k\ell}:=	\{(x_1(k),e^{is_{\ell_1}}x_2(k),\dots,e^{is_{\ell_{n-1}}}x_{n}(k)) \in   \mathbb{R}^{2n-1}\cap S^{2n-1}\}.
		\end{gather*}
We also define a set
\[
 F^n:=\bigcup\limits_{k,\ell}{F}^n_{k\ell}.
\]
Then $S_0(F^n)$ is not normal as $(F^n)'$ is countable. One can proceed as before to show that $v_n=(1,0,\dots,0)\in S^{2n-1}$ generates a nonsparse leaf of $S_0(F^n)$ for each positive integer $n$. Therefore, $S_0(F^n)$ is a countable formal Forelli suspension.
\end{example}

\begin{example}\label{nowhere dense Forelli suspension}
\normalfont
 Let $\{s_{\ell}\}$ be the same sequence as in Example \ref{nowhere dense formal Forelli suspension}. For each positive integer ${\ell}$, define
\begin{align*}
G_{\ell}:=\{(x,e^{is_{\ell}}y)\in \mathbb{R}^3&\cap S^{3}:\,x,y\in \mathbb{R}\}.
\end{align*}
Note that $(G_{\ell})'_1\subset \mathbb{C}$ and $(G_{\ell})'_2\subset \mathbb{C}$ are biholomorphic to the real line $\mathbb{R}=\{z_1\in \mathbb{C}:\,\textup{Im}\,z_1=0\}$. By applying the Phragm$\acute{\textup{e}}$n-Lindel$\ddot{\textup{o}}$f principle for subharmonic functions (see p.33 of \cite{Ransford95}), one can check that $V^{\ast}_{\mathbb{R}}(z)=0$ for any $z\in \mathbb{C}$. Therefore, every point of $G_{\ell}$ generates a regular leaf and in particular, each $S_0(G_{\ell})$ is normal. Note that each $S_0(G_{\ell})$ is sparse as
\[
S_0(G_{\ell})\subset \{(z,w)\in \mathbb{C}^2:\textup{Im}\,(e^{is_{\ell}}z\bar{w})=0\}.
\]
But the suspension generated by $G:=\bigcup\limits_{\ell=1}^{\infty}G_\ell$ provides an example of nowhere dense Forelli suspension as it contains a normal suspension $S_0(G_1)$ and a formal Forelli suspension $S_0(F)$ constructed in Example \ref{nowhere dense formal Forelli suspension}. Note that $v=(1,0)\in S^3$ generates a regular leaf and a nonsparse leaf of $S_0(G)$.

This construction can also be generalized to higher dimensions. For each positive integer $\ell$, define
\[
G^n_{\ell}:=\{(x,z_2,\dots,z_{n-1},e^{is_{\ell}}y)\in \mathbb{R}^{2n-1}\cap S^{2n-1}:\,x,y\in \mathbb{R},\,z_i\in \mathbb{C}\,\forall i\}.
\]
Then each $(G^n_{\ell})'_{1}=\mathbb{C}^{n-2}\times \mathbb{R}\subset \mathbb{C}^{n-1}$ is $L$-regular at every point of itself. So $S_0(G^n_{\ell})$ is normal. Note that  each $S_0(G^n_{\ell})$ is sparse as
\[
S_0(G_{\ell})\subset \{(z_1,\dots,z_n)\in \mathbb{C}^n:\textup{Im}\,(e^{is_{\ell}}z_{1}\,\bar{z}_n)=0\}.
\]
But the suspension generated by $G^n:=\bigcup\limits_{\ell=1}^{\infty}G^n_\ell$ contains the formal Forelli suspension $S_0(F^n)$ constructed in Example $\ref{nowhere dense formal Forelli suspension}$. Therefore, $S_0(G^n)$ is a nowhere dense Forelli suspension.
\end{example}

\vspace{50pt}

Ye-Won Luke Cho (\texttt{ww123hh@pusan.ac.kr}) 

\medskip

BRL for Geometry of Submanifolds, 

Department of Mathematics,

Pusan National University, 

Busan 46241, The Republic of Korea.


\begin{thebibliography}{ABC-99}


\bibitem[Chir06]{Chirka06}
E. M. Chirka, \textit{Variation of Hartogs' theorem}, Proc. Steklov. Inst. 
Math. $\textbf{253}$ (2006), no.2, 212--220.

\bibitem[CK21]{CK21} Y.-W. Cho, K.-T. Kim,  
\textit{Functions holomorphic along a $C^1$ pencil of holomorphic discs}, 
J. Geom. Anal. \textbf{31} (2021), 10634--10637.

\bibitem[For77]{Forelli77} F. Forelli, \textit{Pluriharmonicity in terms 
of harmonic slices}, Math. Scand.
\textbf{41} (1977), no. 2, 358--364.

\bibitem[JKS13]{JKS13} J.-C. Joo, K.-T. Kim, G. Schmalz, 
\textit{A generalization of Forelli's theorem}, Math. Ann. $\textbf{355}$ 
(2013), no. 3, 1171--1176.

\bibitem[JKS16]{JKS16} J.-C. Joo, K.-T. Kim, G. Schmalz, 
\textit{On the generalization of Forelli's theorem}, Math. Ann. 
$\textbf{365}$ (2016), no.3--4, 1187--1200.


\bibitem[KPS09]{KPS09} K.-T. Kim, E. A. Poletsky, G. Schmalz, 
\textit{Functions holomorphic along holomorphic vector fields}, 
J. Geom. Anal. \textbf{19} (2009), no. 3, 655--666.

\bibitem[Kim13]{Kim13} K.-T. Kim,  
\textit{On the generalization of Forelli's theorem - a brief survey}, 
Sandai: Tohoku University Press (2013), 13--23.

\bibitem[Kli91]{Klimek91} M. Klimek, \textit{Pluripotential theory}, 
London Mathematical Society (1991).

\bibitem[Kra18]{Krantz18} S. G. Krantz, \textit{On a theorem of F. Forelli
	and a result of Hartogs}, Complex Var. Elliptic equ. $\textbf{63}$ (2018), no.4, 591--597. 
\bibitem[Lel51]{Lelong51}  P. Lelong, \textit{On a problem of M. A. Zorn}, 
Proc. Amer. Math. Soc. $\textbf{2}$ (1951), 12--19. 


\bibitem[LevM88]{LevenMol88} N. Levenberg, R.E. Molzon, \textit{Convergence sets of a formal power series}. Math. Z., (3) 197 (1988), 411-420.

\bibitem[Nara95]{Nara95} R. Narasimhan, \textit{Several complex variables}, Reissue ed., University of Chicago Press (1995).

\bibitem[Rans95]{Ransford95} T. Ransford, \textit{Potential theory in the complex plane}, London Mathematical Society (1995).


\bibitem[Sadu20]{Sadullaev} A. Sadullaev, \textit{Holomorphic continuation of a formal series along analytic curves}, Complex Var. Elliptic equ. (2020), https://doi.org/10.1080/17476933.2020.1818734.

\bibitem[Shi89]{Shiffman89} B. Shiffman, \textit{Separate analyticity and Hartogs theorems}, Ind. Univ. Math. Journal. $\textbf{38}$ (1989), no. 4, 943--957. 

\bibitem[Sici81]{Siciak81} J. Siciak, \textit{Extremal plurisubharmonic 
functions in $\mathbb{C}^n$}, Ann. Polon. Math. \textbf{39} (1981), 
no. 1, 175--211.




\end{thebibliography}
\end{document}